\documentclass[10pt,twocolumn,twoside]{IEEEtran}
\usepackage{amsmath,amssymb,euscript ,yfonts,psfrag,latexsym,dsfont,graphicx,bbm,color,amstext,wasysym,subfig,parskip,pdfsync}
\usepackage{epstopdf}
\usepackage{flushend}
\usepackage{amsthm}
\usepackage{mathtools}
\usepackage{makecell}
\usepackage{graphicx} 
\usepackage{caption}
\usepackage{subcaption}
\usepackage{amsfonts}
\usepackage{algorithmic}
\usepackage{textcomp}
\usepackage{xcolor,mathdots}
\usepackage{hhline}
\usepackage{multirow}
\usepackage{float} 
\usepackage{url}
\usepackage{psfrag}

\newcounter{plm}

\newtheorem{proposition}[plm]{Proposition}

\begin{document}
\def\spacingset#1{\def\baselinestretch{#1}\small\normalsize}

\title{A general solution to the simultaneous stabilization problem by analytic interpolation}

\author{Yufang Cui,~\IEEEmembership{Student Member,~IEEE} and Anders Lindquist,~\IEEEmembership{Life Fellow,~IEEE}
\thanks{Y.\ Cui is with the Department of Automation, Shanghai Jiao Tong University, Shanghai, China; {email: cui-yufang@sjtu.edu.cn} and A.\ Lindquist is with the Department of  Automation and the School of Mathematical Sciences, Shanghai Jiao Tong University, Shanghai, China; {email: alq@kth.se}}
} 

\maketitle

\setlength{\parindent}{15pt}
\parskip 6pt
\def\spacingset#1{\def\baselinestretch{#1}\small\normalsize}
\newcommand{\E}{\operatorname{E}}
\newcommand{\bX}{\mathbf X}
\newcommand{\bH}{\mathbf H}
\newcommand{\bA}{\mathbf A}
\newcommand{\bB}{\mathbf B}
\newcommand{\bN}{\mathbf N}

\newcommand{\bHo}{{\stackrel{\circ}{\mathbf H}}}

\newcommand{\bHom}{{\stackrel{\circ\,\,}{{\mathbf H}_{t}}}{\hspace*{-13pt}\phantom{H}}^-}
\newcommand{\bHop}{{\stackrel{\circ\,\,}{{\mathbf H}_{t}}}{\hspace*{-13pt}\phantom{H}}^+}
\newcommand{\bHotwo}{{\stackrel{\circ\,\,\,\,\,}{{\mathbf H}_{t_2}}}{\hspace*{-16pt}\phantom{H}}^-}
\newcommand{\bHomn}{{\stackrel{\circ}{{\mathbf H}}}{\hspace*{-9pt}\phantom{H}}^-}
\newcommand{\bHopn}{{\stackrel{\circ}{{\mathbf H}}}{\hspace*{-9pt}\phantom{H}}^+}

\newcommand{\EbHom}{\E^{{\stackrel{\circ\,\,}{{\mathbf H}_{t}}}{\hspace*{-11pt}\phantom{H}}^-}}
\newcommand{\EbHop}{\E^{{\stackrel{\circ\,\,}{{\mathbf H}_{t}}}{\hspace*{-10pt}\phantom{H}}^+}}

\spacingset{.97}

\begin{abstract} 
	
In this paper, we tackle the significant challenge of simultaneous stabilization in control systems engineering, where the aim is to employ a single controller to ensure stability across multiple systems. We delve into both scalar and multivariable scenarios. For the scalar case, we present the  necessary and sufficient conditions for a single controller to stabilize multiple plants and reformulate these conditions to interpolation constraints, which expand Ghosh's results by allowing derivative constraints. Furthermore, we implement a methodology based on a Riccati-type matrix equation, called the Covariance Extension Equation. This approach  enables us to parameterize all potential solutions using a monic Schur polynomial. Consequently, we extend our result to the multivariable scenario and derive the necessary and sufficient conditions for a group of $m\times m$ plants to be simultaneously stabilizable, which can also be solved by our analytic interpolation method. Finally, we construct four numerical examples, showcasing the application of our method across various scenarios encountered in control systems engineering and highlighting its ability to stabilize diverse systems efficiently and reliably.

\end{abstract}


\newcommand{\mR}{{\mathbb R}}
\newcommand{\mZ}{{\mathbb Z}}
\newcommand{\mN}{{\mathbb N}}
\newcommand{\mE}{{\mathbb E}}
\newcommand{\mC}{{\mathbb C}}
\newcommand{\mD}{{\mathbb D}}
\newcommand{\bU}{{\mathbf U}}
\newcommand{\bW}{{\mathbf W}}
\newcommand{\cF}{{\mathcal F}}

\newcommand{\trace}{{\rm trace}}
\newcommand{\rank}{{\rm rank}}
\newcommand{\Real}{{\Re}e\,}
\newcommand{\half}{{\frac12}}

\spacingset{1}

\section{Introduction}\label{Introduction}

In order to introduce and analyze the problems considered in this paper we need the following notation.
\begin{align*}
&\mathbb{C} : \text{the complex plane} \\
&\mathbb{R} : \text{the real line} \\
&\mathbb{C}_{-} : \text{open left half of the complex plane} \\
&\mathbb{C}_{+}: \left[ \mathbb{C} - \mathbb{C}_{-} \right] \cup \{\infty\} \\
&\mathbb{D}: \text{interior of the unit disc, i.e. } \{z\mid |z|<1\}\\
&\mathbb{H} : \text{ring of proper rational functions with real}\\
&\qquad\text{co-efficients with poles in } \mathbb{C}_{-} \\
&\mathbb{H}^{p \times m} : \text{set of } p \times m \text{ matrices whose elements belong to } \mathbb{H} \\
&\mathbb{J} : \text{set of multiplicative units in } \mathbb{H} 
\end{align*}

In the field of control systems engineering, ensuring that multiple systems remain stable using only one controller is a significant and practical challenge \cite{Blondel, Nett, Chaoui}. This challenge is known as simultaneous stabilization, and it's an important area of study that has received considerable attention from both theoretical researchers and practitioners \cite{Smith, Wang,Minto,Saif}. 

The essence of simultaneous stabilization lies in developing a single controller capable of stabilizing a variety of plants \cite{Vidyasagar,Youla}. This paper addresses both Single-Input Single-Output (SISO) and Multiple-Input Multiple-Output (MIMO) scenarios. Specifically, for the SISO case, our objective is to find a suitable compensator $k(s)$ that stabilizes a family of SISO proper transfer functions $p_\lambda(s)$, which defined as
\begin{equation}\label{systems}
p_\lambda(s) = \frac{\lambda x_1(s) + (1 - \lambda) x_0(s)}{\lambda y_1(s) + (1 - \lambda) y_0(s)},
\end{equation}
where $\lambda \in [0,1]$ and $x_0(s), x_1(s), y_0(s), y_1(s) \in \mathbb{H}$. 

Expanding to the MIMO case, the problem involves finding an appropriate $m \times m$ compensator $K(s)$ that can stabilize a family of systems $P_\lambda(s)$, defined as:
\begin{equation}\label{MIMO} P_\lambda(s)=N_\lambda(s)D_\lambda(s)^{-1}
\end{equation}
where $\lambda\in[0,1]$ and 
\begin{equation}
    N_\lambda=\lambda N_1+(1-\lambda)N_0
\end{equation}
\begin{equation}
    D_\lambda=\lambda D_1+(1-\lambda)D_0
\end{equation}
where $N_0,N_1,D_0,D_1 \in \mathbb{H}^{m \times m}$.

The concept of simultaneous stabilization emerged as an extension of classical control theory, where the primary focus was on stabilizing single systems. The challenge of stabilizing multiple systems with a single controller was first addressed in the late 20th century.  Researchers like Vidyasagar, Saeks, and Murray \cite{Vidyasagar,Vidyasagar1,Youla,Saeks} provided early theoretical frameworks for simultaneous stabilization. Their work introduced critical concepts and conditions under which a single controller could stabilize multiple plants. This period saw the development of various methods \cite{Werner,Blondel1,Xu}, including algebraic criteria and geometric approaches, which significantly advanced the field.

BK Ghosh's work \cite{Ghosh,Ghosh2,Ghosh3} was pivotal in introducing interpolation methods to the problem of simultaneous stabilization. By recasting the problem into a Nevanlinna-Pick interpolation framework, Ghosh provided a new lens through which to view and solve stabilization challenges. 

In this paper, we build upon the foundational work of BK Ghosh, who effectively transformed the scalar case simultaneous stabilization problem into a Nevanlinna-Pick interpolation problem. We employ a more comprehensive interpolation approach based on our prior research on a Riccati-type method for analytic interpolation \cite{CLtac, CLcdc}, which is built upon algorithms for the partial stochastic realization problem \cite{b2,BGuL,b1,BLpartial} and on \cite{b6}. This approach  enables us to solve situations with derivative constraints and parameterize all potential solutions using a matrix polynomial.

The paper's structure is as follows. In Section~\ref{sec:3condition}, we present necessary and sufficient conditions for a family of SISO plants to be simultaneous stabilizable. Section~\ref{sec:interpolation} demonstrates the conversion of the SISO simultaneous stabilization problem into an analytic interpolation problem. In Section~\ref{sec:multivariable}, we outline the essential conditions
necessary for a group of MIMO plants to be simultaneously stabilizable and illustrate how the MIMO simultaneous stabilization problem can be converted into a multivariable analytic interpolation problem. Section~\ref{sec:CEE} focuses on solving the analytic interpolation problem using the Covariance Extension Equation (CEE). In Section~\ref{sec:applications}, we provide some
simulations to show how the method can be applied to
stabilize multiple plants. Finally, in Section~\ref{sec:conclusion} we give some
concluding remarks and suggestions for future research.

\section{Simultaneous stabilization in the scalar case}\label{sec:3condition}

To address the SISO simultaneous stabilization problem, we firstly consider a straightforward scenario involving two distinct plants characterized by coprime factorizations, denoted as follows:
\begin{equation}
\label{p0p1}
p_0(s)=\frac{x_0(s)}{y_0(s)}, \qquad p_1(s)=\frac{x_1(s)}{y_1(s)},
\end{equation}
where $x_{i}(s), y_{i}(s) \in\mathbb{H} $ and $y_{i}(s)$ is proper but not strictly proper. The objective is to identify a proper compensator capable of simultaneously stabilizing both $p_0$ and $p_1$.

\begin{proposition}\label{prop1}
The pair of distinct plants $p_0$ and $p_1$ is simultaneously stabilized by a proper compensator if and only if there exists $\delta_{0}(s), \delta_{1}(s) \in \mathbb{J}$, satisfying the following conditions:

(i)  If $x_{0} y_{1}-x_{1} y_{0}$ has zeros at $s_{1},\cdots, s_{t}$ in $\mathbb{C_+}$ with  multiplicities $m_{1}, \cdots, m_{t}$ respectively, then it is necessary that $s_{1},\cdots, s_{t}$ are also the zeros of $\delta_{0} y_{1}-\delta_{1} y_{0}$ and $\delta_{1} x_{0}-\delta_{0} x_{1}$ with multiplicities not less than $m_{1}, \cdots, m_{t}$, respectively.

(ii) If $x_{0} y_{1}-x_{1} y_{0}=0$ at $\infty$ with multiplicity $m_{\infty}$, then $\delta_{1} x_{0}-\delta_{0} x_{1}=0$ at $\infty$ with the same multiplicity $m_{\infty}$.
\end{proposition}

\medskip

\begin{proof}

The proof of Proposition 1 is based on BK Ghosh's theory \cite{Ghosh,Ghosh2}. Denote the required  compensator by coprime factorization as 
\begin{equation}
\label{k}
k(s)=\frac{x_c(s)}{y_c(s)},
\end{equation}
where $x_{c}(s),y_{c}(s)\in \mathbb{H}$ and $y_{c}(s)$ is proper but not strictly proper. Consequently, the transfer function for the closed-loop system can be expressed as
\begin{equation}
G_{i}(s)=\frac{n_{i}(s)}{d_{i}(s)}=\frac{x_i(s)y_c(s)}{y_i(s)y_c(s)+x_i(s)x_c(s)},\quad i={0,1}
\end{equation}

Given that $x_i(s),x_c(s),y_i(s),y_c(s)\in \mathbb{H}$, indicating that their poles are situated in $\mathbb{C_-}$, it follows that the poles of $d_{i}(s)$ and $n_{i}(s)$  also reside in $\mathbb{C_-}$. To achieve the simultaneous stabilization of $p_0(s)$ and $p_1(s)$, all zeros of $d_{i}(s)$ must be located in $\mathbb{C_-}$.
	
	If  $p_i(s)$ is proper but not strictly proper, then the numerator $n_{i}(s)$ will have a non-zero value at $\infty$. In order for the function $G_{i}(s)$ to be proper, the denominator $d_{i}(s)$ must also have a non-zero value at $\infty$.

In the case where $x_i(s)/y_i(s)$ is strictly proper, the denominator $d_{i}(s)$ will have a non-zero value at $\infty$, resulting in $G_{i}(s)$ being strictly proper.

Hence, the simultaneous stabilization of $p_0$ and $p_1$ depends on the presence of $\delta_{0}$ and $\delta_{1}$ in  $\mathbb{J}$, such that
	\begin{equation}\label{cond1}
		x_{i}(s) x_{c}(s)+y_{i}(s) y_{c}(s)=\delta_{i}(s),\quad i={0,1}
	\end{equation}		
	By solving equation \eqref{cond1} for $x_{c}$ and $y_{c}$, we obtain
	\begin{equation}
		\label{8}
		\begin{aligned}
			&x_{c}(s)=(\delta_{0} y_{1}-\delta_{1} y_{0})/(x_{0} y_{1}-x_{1} y_{0}) \\
			&y_{c}(s)=(\delta_{1} x_{0}-\delta_{0} x_{1}) /(x_{0} y_{1}-x_{1} y_{0})
		\end{aligned}
	\end{equation}
	
	Condition (i) is  necessary and sufficient for $x_{c}(s)\in \mathbb{H}$ and $y_{c}(s)\in\mathbb{H}$. Condition (ii) is necessary and sufficient for $y_c(s)\in \mathbb{H}$ to be proper but not strictly proper. It is important to note that $x_c(\infty)/y_c(\infty)\neq\infty$, indicating that $x_{c}(s) / y_{c}(s)$ represents a proper rational function.
\end{proof}
\medskip
Next, we consider the broader scenario. Denote two distinct plants as \eqref{p0p1} and consider
\begin{equation}
	p_\lambda(s)=\frac{x_\lambda(s)}{y_\lambda(s)}=\frac{\lambda x_1(s)+(1-\lambda)x_0(s)}{\lambda y_1(s)+(1-\lambda)y_0(s)}
\end{equation}
where $\lambda \in[0,1]$. 

\begin{proposition}

A proper compensator can stabilize the set of plants represented by $p_\lambda(s)$ for $\lambda\in[0,1]$  if and only if there are  $\delta_{0} \in \mathbb{J}$ and $\delta_{1}\in\mathbb{J}$  that satisfy the conditions described in Proposition 1, along with an extra condition showed below.
\end{proposition}

Condition (iii): $\delta_{1}/\delta_{0}$ should not have any intersection with the nonpositive real axis, including $\infty$ at any point in $\mathbb{C_+}$.

\medskip

\begin{proof}
To stabilize the two plants \eqref{p0p1} simultaneously, the compensator \eqref{k} must meet conditions (i) and (ii). Furthermore, for \eqref{k} to stabilize all other plants $x_{\lambda} / y_{\lambda}$ simultaneously, it is necessary and sufficient that there are $\delta_{\lambda}\in\mathbb{J}$ , for $\lambda\in (0,1)$ such that the equation
\begin{equation}\label{cond2}
x_{c} x_{\lambda}+y_{c} y_{\lambda}=\delta_{\lambda}
\end{equation}
holds. The combination of \eqref{cond1} and \eqref{cond2} yields
\begin{equation}
\lambda\delta_1+(1-\lambda)\delta_{0}=\delta_{\lambda},\lambda\in(0,1)
\end{equation}
Given that the poles of $\delta_{0}$ and $\delta_{1}$ reside in $\mathbb{C_-}$, it follows that the poles of $\delta_{\lambda}$ also belong to $\mathbb{C_-}$. To ensure $\delta_\lambda$ is in $\mathbb{J}$, all zeros of $\delta_\lambda$ must situate in $\mathbb{C_-}$.
From
	\begin{equation}
		\lambda\delta_1+(1-\lambda)\delta_{0}=\delta_{\lambda}=0
	\end{equation}
	we get 
	\begin{equation}
		\frac{\delta_{1}}{\delta_{0}}=1-\frac{1}{\lambda} \in (-\infty,0), \text{for~} \lambda \in (0,1)
	\end{equation}
	
	If $\delta_{1}/\delta_{0}$ intersects  $(-\infty,0)$ at a point in $\mathbb{C_+}$ (denoted as $\hat{s}$), then there exists a  $\lambda$ within the range of $(0,1)$ that satisfies
	\begin{equation}
		\frac{\delta_{1}}{\delta_{0}}(\hat{s})=1-\frac{1}{\lambda}
	\end{equation}
	and 
	\begin{equation}
		\delta_{\lambda}(\hat{s})=0
	\end{equation}
	This implies that $\delta_{\lambda}$ is not in $\mathbb{J}$, which contradicts the requirement that all $\delta_{\lambda}$ belong to $\mathbb{J}$. Therefore, for $\delta_\lambda\in\mathbb{J}, \lambda\in(0,1)$ to hold, it is necessary that $\delta_1/\delta_0$ does not have any intersection with  $(-\infty,0)$ at $\mathbb{C_+}$. Furthermore, since $\delta_0\in\mathbb{J}$ and $\delta_1\in \mathbb{J}$, then Proposition 2 follows.	
\end{proof}

\section{Transformation to the analytic interpolation problem}\label{sec:interpolation}

In this section, we demonstrate the process of converting the SISO simultaneous stabilization problem into an analytic interpolation problem,  which is defined in the following way.  Find a real rational Carath\'eodory function $F$ of size $\ell\times \ell$, i.e., a function $F$ that is analytic in $\mathbb{D}$, and satisfies the inequality condition
\begin{equation}
\label{F+F*}
F(z)+F(z)^{*}  > 0, \quad z\in\mathbb{D},
\end{equation}
and also fulfills the interpolation conditions
\begin{align}
\label{interpolation}
 \frac{1}{j!}F^{(j)}(z_{k})=W_{kj},\quad& k=1,\cdots,m,   \\
    &   j=0,\cdots n_{k}-1 ,\notag
\end{align}
where $'$ denotes transposition, $F^{(j)}(z)$ is the $j$th derivative of $F(z)$, and $z_1,\dots,z_m$ are distinct points in $\mathbb{D}$ and $W_{kj}\in\mathbb{C}^{\ell\times\ell}$ for each $(k,j)$. The complexity of the rational function $F(z)$ is constrained by limiting its McMillan degree to be at most $\ell n$, where
\begin{equation}
\label{deg(f)}
n=\sum_{k=1}^{m}n_k -1 .
\end{equation}

Firstly, the conditions discussed in Proposition~\ref{prop1} can be reformulated as interpolation conditions showed below. 
\begin{proposition}
	If $s_{j}$ is a zero of  $x_{0} y_{1}-x_{1} y_{0}$ in $\mathbb{C_+}$ with multiplicity $n+1$, while not a zero of $y_0$ and $y_1$, or $x_0$ and $x_1$, then Proposition \ref{prop1} means that the $i$-th derivative of $\delta_{1}(s) / \delta_{0}(s)$ satisfying the following interpolation constraints
	\begin{equation}
		\label{prop3}
		(\frac{\delta_{1}}{\delta_{0}})^{(i)}(s_j)=(\frac{y_{1}}{y_{0}})^{(i)}(s_j)
	\end{equation}
	where $\delta_{1}, \delta_{0} \in \mathbb{J}$, $i=0,\cdots,n$.
\end{proposition}
\medskip
\begin{proof}
	We first introduce the Leibniz formula \cite{Olver}, which states that if $u$ and $v$ are n-times differentiable functions, then the product $uv$ is also n-times differentiable and the n-th derivative is 
	\begin{equation}
		(u v)^{(n)}=\sum_{k=0}^n C_n^k u^{(n-k)} v^{(k)},
	\end{equation}
 
 Since $s_{j}$ is a zero of $x_{0} y_{1}-x_{1} y_{0}$ of multiplicity $n+1$, 
	\begin{equation}
		\label{xy}
		(x_{0} y_{1}-x_{1} y_{0})^{(i)}(s_j)=0,\quad i=0,\cdots n.
	\end{equation}
	
	By Proposition \ref{prop1}, we need to have
	\begin{equation}
		\label{derivative}
		(\delta_{0}y_1-\delta_{1}y_0)^{(i)}(s_j)=0,\quad i=0,\cdots n
	\end{equation}
	
	which is equivalent to
	\begin{equation}
		\label{equal}
		(y_1-\frac{\delta_{1}}{\delta_{0}}y_0)^{(i)}(s_j)=0,\quad i=0,\cdots n.
	\end{equation}
	
	By Leibniz formula, \eqref{equal} implies
	\begin{equation}
		(y_1)^{(i)}(s_j)-\sum_{k=0}^{i} C_i^k (\frac{\delta_{1}}{\delta_{0}})^{(i-k)}(s_j) y_0^{(k)}(s_j)=0
	\end{equation}
	for $i=0,\cdots,n$. Thus, if $n=0$,
	\begin{equation}
		\frac{\delta_{1}}{\delta_{0}}(s_j)=\frac{y_{1}}{y_{0}}(s_j),
	\end{equation}
	Assume that \eqref{prop3} holds for $i=0,\cdots,n-1$ with multiplicity $n$. Then for multiplicity $n+1$, we need an extra constraint
	\begin{equation}
		(y_1-\frac{\delta_{1}}{\delta_{0}}y_0)^{(n)}(s_j)=0
	\end{equation}
	which is 
	\begin{equation}
		(y_1)^{(n)}-\sum_{k=1}^{n} C_n^k (\frac{y_{1}}{y_{0}})^{(n-k)} y_0^{(k)}-(\frac{\delta_{1}}{\delta_{0}})^{(n)}y_0=0
	\end{equation}
	at $s_j$. By Leibniz formula,
	\begin{equation}
		(y_1)^{(n)}-(y_1)^{(n)}+(\frac{y_1}{y_0})^{(n)}y_0-(\frac{\delta_{1}}{\delta_{0}})^{(n)}y_0=0
	\end{equation}
	at $s_j$, yielding
	\begin{equation}
		(\frac{\delta_{1}}{\delta_{0}})^{(n)}(s_j)=(\frac{y_1}{y_0})^{(n)}(s_j).
	\end{equation}
	Then, Proposition 3 can be derived using mathematical induction.
	Similarly, \eqref{xy} means
	\begin{equation}
		(\frac{x_{1}}{x_{0}})^{(i)}(s_j)=(\frac{y_1}{y_0})^{(i)}(s_j)\quad i=0,1,\cdots,n,
	\end{equation}
	so
	\begin{equation}
		(\frac{\delta_{1}}{\delta_{0}})^{(i)}(s_j)=(\frac{y_1}{y_0})^{(i)}(s_j)=(\frac{x_{1}}{x_{0}})^{(i)}(s_j),
	\end{equation}
	for $i=0,\cdots,n$ which concludes the proof
\end{proof}
\medskip
If $s_{j}$ is a zero of $ y_{1}$ and $ y_{0}$, or $ x_{1}$ and $ x_{0}$ with a certain multiplicity, then the interpolation constraints in Proposition 3 need change. For instance, if $s_{j}$ is a zero of $ x_{0}$ and $x_{1}$ with multiplicity 1, but not a zero of $y_0$ and $y_1$, then it is required that $\delta_{1} / \delta_{0}$ interpolates  $(s_{j},(y_{1} / y_{0})(s_{j))}$.

Proposition 2 implies that  $\frac{\delta_{1}}{\delta_{0}}$ maps $\mathbb{C_+}$ to $re^{i\theta}$, where $r\in (0,\infty)$  and $\theta\in(-\pi,\pi)$.

In order to convert to analytic interpolation problem, define
\[ F(s):=\sqrt{\frac{\delta_{1}}{\delta_{0}}} \]
which maps  $\mathbb{C_+}$ to the open right half-plane, i.e., to $\sqrt{r}e^{i\theta/2}$.

By employing the Möbius transformation given by 
\begin{equation}\label{mob}
    z=(1-s)(1+s)^{-1}
\end{equation}
 which maps $\mathbb{C_+}$ into $\mathbb{D}$, we set
\begin{equation}
	f(z):=F((1-z)(1+z)^{-1})
\end{equation}
then the problem becomes identifying a Carathéodory function $f(z)$ that fulfills specific interpolation constraints. This is an analytic interpolation problem, after solving $f(z)$, the compensator $k(s)$ can be determined by the following procedures.
\begin{equation}
	F(s)=f((1-s)(1+s)^{-1})
\end{equation}
\begin{equation}
	k(s)=\frac{F^2x_0-x_1}{y_1-F^2x_1}
\end{equation}

Overall, there exist an infinite number of solutions to this problem. However, as detailed in Section~\ref{sec:CEE}, these solutions can be fully characterized by a parameter denoted as $\Sigma(z)$. By freely selecting $\Sigma(z)$, we have the flexibility to adjust the solution according to specific requirements.

\section{Simultaneous stabilization in the multivariable case}\label{sec:multivariable}
In this section, we will generalize our results from SISO simultaneous stabilization problem to multivariable case. As explained in \cite{Ghosh,Ghosh2,Ghosh3}, every SISO system can be written as $x(s) / y(s)$, where $x(s), y(s) \in \mathbb{H}$, and an $m \times m$  plant $P(s)$ has the left coprime representation  $D_l(s)^{-1} N_l(s)$ and the right coprime representation $N_r(s)D_r(s)^{-1}$, where $N_l(s),D_l(s),N_r(s),D_r(s) \in \mathbb{H}^{m \times m}$.

Similar to SISO case, we firstly consider a simple case: Given two different plants 
\begin{equation}
  P_i(s)= N_i(s)D_i(s)^{-1}, \quad i=0, 1  
\end{equation}
where $N_i(s) \in \mathbb{H}^{m \times m}, D_i(s) \in \mathbb{H}^{m \times m}$, for $i=0,1$, find a proper $m\times m$ compensator $K(s)$ which can stabilize $P_0(s)$ and $P_1(s)$.

Set
\begin{equation}
  M(s)=
    \begin{bmatrix}
	N_0(s) & N_1(s) \\
	D_0(s) & D_1(s)
\end{bmatrix},
\end{equation}
and let $\text{Adj}(M(s))$ be its adjoint matrix. Assume generically that  $\det(M(s))$ has simple zeros in $\mathbb{C_+}$ at $s_1,\cdots,s_t$ and $\det(M(\infty))\neq 0$.
\begin{proposition}
	The two plants $P_0,P_1$ can be simultaneously stabilized by a proper compensator if and only if there exists $\Delta_i(s) \in \mathbb{H}^{m \times m}$, $\operatorname{det} \Delta_i(s) \in \mathbb{J}, i=0, 1$ , such that if $s_{1},s_2$, $\cdots, s_{t}$ are the zeros of $\det(M)$ in $\mathbb{C_+}$, then 
	\begin{equation}
	    \begin{bmatrix}
		\Delta_0(s)&\Delta_1(s)
	\end{bmatrix}\text{Adj}(M(s))=\mathbf{0}
	\end{equation}
	at $s_{1},s_2$, $\cdots, s_{t}$.	
\end{proposition}
\begin{proof}
 Let us represent the compensator as
	\begin{equation}
	    K(s)=D_c(s)^{-1}N_c(s) 
	\end{equation}
	where $N_c(s) \in \mathbb{H}^{m \times m}, D_c(s) \in \mathbb{H}^{m \times m}, N_c(s), D_c(s)$ are coprime. Then the compensator stabilizes $P_0(s),P_1(s)$ if and only if
 \begin{equation}
     \quad N_c(s)N_i(s) +D_c(s)D_i(s) =\Delta_i(s)
 \end{equation}
	for some $\Delta_i(s) \in \mathbb{H}^{m \times m}$, $\operatorname{det} \Delta_i(s)\in \mathbb{J}, i=0, 1$ respectively \cite{Ghosh3}. 
	We can write this in matrix form as
	\begin{equation}
	    \begin{bmatrix}
	        N_c&D_c
	    \end{bmatrix}
		\begin{bmatrix}
		    N_0(s) & N_1(s) \\
		D_0(s) & D_1(s)
		\end{bmatrix}
	=\begin{bmatrix}
		\Delta_0(s) ~\Delta_1(s)
	\end{bmatrix},
	\end{equation}
   and  then
	 \begin{equation}
	     \begin{bmatrix}
	        N_c&D_c
	    \end{bmatrix}=\begin{bmatrix}
		\Delta_0(s) ~\Delta_1(s)
	\end{bmatrix}\frac{\text{Adj}(M(s))}{\det(M(s))}.
	 \end{equation}
	 In order to ensure that $N_c(s) \in \mathbb{H}^{m \times m}, D_c(s) \in \mathbb{H}^{m \times m}$, it is necessary and sufficient that 
	\begin{equation}
		\label{matrx_cond}
			[\begin{matrix}
			\Delta_0(s)~\Delta_1(s)
		\end{matrix}]\text{Adj}(M(s))=\mathbf{0}		
	\end{equation}
at $s_{1},s_2$, $\cdots, s_{t}$, where $s_{1},s_2$, $\cdots, s_{t}$  are the simple zeros of $\det(M)$ in $\mathbb{C_+}$.
\end{proof}

Assume generically that $\text{Adj}(M(s_i))$ are of rank 1 at $s_{1},s_2$, $\cdots, s_{t}$, If $\text{Adj}(M(s_i))$ are spanned by a column vector $\mathbf{r_i}=[\mathbf{r_{i1}},\mathbf{r_{i2}}]'$, then
\begin{equation}\label{delta}
    \Delta_0(s_i)\mathbf{r_{i1}}'+\Delta_1(s_i)\mathbf{r_{i2}}'=\mathbf{0}
\end{equation}
for $i=1, \cdots, t$. From equation \eqref{delta}, we can derive $(\Delta_0(s_i))^{-1}\Delta_1(s_i)$.

Now we consider the more general case, suppose
\begin{equation}
    P_\lambda(s)=N_\lambda(s)D_\lambda(s)^{-1}
\end{equation}
where $\lambda\in[0,1]$ and 
\begin{equation}
    N_\lambda=\lambda N_1+(1-\lambda)N_0
\end{equation}
\begin{equation}
    D_\lambda=\lambda D_1+(1-\lambda)D_0
\end{equation}

\begin{proposition}

If $(\Delta_0)^{-1}\Delta_1$ is diagonalizable at $s\in \mathbb{C}_+$, then the set of plants $P_\lambda(s)$ for $\lambda\in[0,1]$ can be simultaneously stabilized by a proper compensator if and only if there exist $\Delta_i(s) \in \mathbb{H}^{m \times m}$, with $\operatorname{det} \Delta_i(s) \in \mathbb{J}$ for $i=0, 1$, satisfying the condition of Proposition 4, along with the additional condition that the eigenvalues of $\Delta_0(s)^{-1}\Delta_1(s)$ do not intersect the nonpositive real axis including infinity  at any point in $\mathbb{C}_+$.

\end{proposition}

\begin{proof}
    Suppose $K(s)$ is the required compensator. To stabilize the plants $P_0(s),P_1(s)$ simultaneously, a necessary and sufficient condition is given by Proposition 4. Additionally, $K(s)$ simultaneously stabilizes every other plants $P_\lambda(s)$ for $\lambda \in (0,1) $ if and only if there exist $\Delta_\lambda(s) \in \mathbb{H}^{m \times m}$, $\operatorname{det} \Delta_\lambda(s) \in \mathbb{J}$ such that
    \begin{equation}
        \quad N_c(s)N_\lambda(s) +D_c(s)D_\lambda(s) =\Delta_\lambda(s), \lambda\in (0,1).
    \end{equation}
    By calculation,
    \begin{equation}
        \lambda\Delta_{1}(s)+(1-\lambda)\Delta_{0}(s)=\Delta_{\lambda}(s).
    \end{equation}

Since $\Delta_0(s) \in \mathbb{H}^{m \times m}$ and $\Delta_1(s) \in \mathbb{H}^{m \times m}$, we have  $\Delta_\lambda(s) \in \mathbb{H}^{m \times m},\lambda\in(0,1)$. In order to ensure $\operatorname{det} \Delta_\lambda(s) \in \mathbb{J}$, we need 
\begin{equation}\label{eigenvalues}
    \operatorname{det} (\lambda\Delta_{1}(s)+(1-\lambda)\Delta_{0}(s))\neq 0, \lambda\in(0,1).
\end{equation}
at $s\in\mathbb{C}_+$. Note that over the complex numbers $\mathbb{C}$, almost every matrix is diagonalizable. Here we suppose $(\Delta_0)^{-1}\Delta_1$ is diagonalizable at $s\in \mathbb{C}_+$ and suppose $\alpha_1(s),\alpha_2(s),\cdots,\alpha_m(s)$ are its eigenvalues, then  \eqref{eigenvalues} means $\alpha_i(s)\neq (1-1/\lambda)$ for all $\lambda\in(0,1)$
at $s\in \mathbb{C}_+$. Also due to $\operatorname{det}(\Delta_0)\in\mathbb{J}$ and $\operatorname{det}(\Delta_1)\in\mathbb{J}$, we can conclude that to ensure simultaneous stabilization, the eigenvalues of $\Delta_0(s)^{-1}\Delta_1(s)$ can not  intersect the nonpositive real axis including infinity  at $s\in \mathbb{C}_+$. 
\end{proof}

Next, we reformulate the MIMO simultaneous stabilization problem as multivariable analytic interpolation problem.

From Proposition 4, we can get the interpolation constraints
\begin{equation}
    \Delta_0(s_i)^{-1}\Delta_1(s_i)=M_i,\quad i=1,\cdots,t.
\end{equation}
Let 
\begin{equation}
   F_1(s)=(\Delta_0(s)^{-1}\Delta_1(s))^{1/2}
\end{equation}
be the square-root of $\Delta_0(s)^{-1}\Delta_1(s)$,  which means $F_1(s)$ need satisfy
\begin{equation}
    F_1(s_i)=M_i^{1/2},\quad i=1,\cdots,t.
\end{equation}
Using the M{\"o}bius transformation defined by
\begin{equation}\label{Mobius}
    z=\frac{1-s}{1+s}.
\end{equation}
which can map $\mathbb{C_+}$ into $\mathbb{D}$, we then set
\begin{equation}
	F(z):=F_1((1-z)(1+z)^{-1})
\end{equation}
which is obviously analytic in $\mathbb{D}$. The following proposition is straightforward.
\begin{proposition}
    If 
    \begin{equation}
        M_i^{1/2}+(M_i^{1/2})^{*}>0,i=1,\cdots,t,
    \end{equation}
    the simultaneous stabilization problem \eqref{MIMO} is simplified to identifying a Carath\'eodory function $F(z)$ such that  the interpolation constraints 
    \begin{equation}
        F(z_i)=M_i^{1/2},i=1,\cdots,t
    \end{equation}
    is satisfied, where $z_i=(1-s_i)/(1+s_i)$. This is a matrix case analytic interpolation problem \eqref{interpolation} when $m=t,n_1=\cdots=n_m=1$. 
\end{proposition}

 After obtaining the solution for $F(z)$, we can apply the following transformations to obtain the compensator $K(s)$.
Denote
\begin{equation}
    \frac{Adj(M(s))}{\det(M(s))}=\begin{bmatrix}
        m_{11}(s) & m_{12}(s) \\
	m_{21}(s) & m_{22}(s)
    \end{bmatrix},
\end{equation}
then
\begin{equation}
    F_1(s)=F((1-s)(1+s)^{-1})
\end{equation}
\begin{equation}
    K(s)=D_c^{-1}N_c=(m_{12}+F_1^2m_{22})^{-1}(m_{11}+F_1^2m_{21})
\end{equation}

\section{Solutions to the analytic interpolation problem}\label{sec:CEE}
This section illustrates the approach to address the analytic interpolation problem \eqref{interpolation} employing the Covariance Extension Equation \cite{CLccdc,CLtac,CLcdc}. Without loss of generality, we set $z_{1}=0$ and $W_{1}=\frac{1}{2}I$. Consequently, $F(z)$ can be expressed as:
 \begin{equation}
\label{ }
F(z)=\tfrac12 I + zH(I-zF)^{-1}G,
\end{equation}
where matrices $H\in\mathbb{R}^{\ell\times\ell n}$, $F\in\mathbb{R}^{\ell n\times\ell n}$, $G\in\mathbb{R}^{\ell n\times\ell}$, and all eigenvalues of matrix $F$ reside in $\mathbb{D}$, $(H,F)$ forms an observable pair.

By defining $\Phi_+(z):=F(z^{-1})$, we obtain:
\begin{equation}
\label{ }
\Phi_+(z)=\tfrac12 I + H(zI-F)^{-1}G,
\end{equation}
which has its poles within the unit disc $\mathbb{D}$. Based on \eqref{F+F*}, it follows that:
\begin{equation}
\Phi_+(e^{i\theta})+\Phi_+(e^{-i\theta})'>0, \quad -\pi\leq \theta\leq\pi ,
\end{equation}
and thus $\Phi_+(z)$ is positive real \cite[Chapter 6]{LPbook}. The problem is then reduced to identifying a rational positive real function $\Phi_+(z)$ of degree at most $\ell n$ that meets the interpolation constraints \eqref{interpolation}. By a coordinate transformation $(H,F,G)\to(HT^{-1},TFT^{-1},TG)$, we can select $(H,F)$ in the observer canonical form
\begin{equation}
    H=\text{diag}(h_{t_1},h_{t_2},\dots,h_{t_\ell}) \in \mathbb{R}^{\ell\times n\ell}
\end{equation}
with $h_\nu:=(1,0,\dots,0)\in\mathbb{R}^\nu$, and
\begin{equation}
\label{F}
F=J-AH \in\mathbb{R}^{n\ell\times n\ell}
\end{equation}
where $A\in\mathbb{R}^{n\ell\times \ell}$, $J:=\text{diag}(J_{t_1},J_{t_2},\dots, J_{t_\ell})$ with $J_\nu$ the $\nu\times\nu$ shift matrix
\begin{equation}
J_\nu =\begin{bmatrix}0&1&0&\dots&0\\0&0&1&\dots&0\\\vdots&\vdots&\vdots&\ddots&0\\
0&0&0&\dots&1\\0&0&0&\dots&0\end{bmatrix}.
\end{equation}
Here $t_1,t_2,\dots,t_\ell$ are the {\em observability indices\/} of $\Phi_+(z)$, and 
\begin{equation}
\label{tsum}
t_1+t_2+\dots+t_\ell=n\ell.
\end{equation}
Next define $\Pi(z):=\text{diag}(\pi_{t_1}(z),\pi_{t_2}(z),\dots,\pi_{t_\ell}(z))$, where $\pi_\nu(z)=(z^{\nu-1},\dots,z,1)$, and the $\ell\times\ell$ matrix polynomial
\begin{equation}
\label{A(z)}
A(z)=D(z) +\Pi(z)A,
\end{equation}
where 
\begin{equation}
\label{D(z)}
D(z):=\text{diag}(z^{t_1},z^{t_2},\dots, z^{t_\ell}).
\end{equation}

From Lemma 1 in \cite{CLcdc},
\begin{equation}\label{lemma}
    H(zI-F)^{-1}=A(z)^{-1}\Pi(z),
\end{equation}
and consequently
 \begin{equation}
\label{AinvB}
\Phi_+(z)=\tfrac12 A(z)^{-1}B(z),
\end{equation}
where
\begin{displaymath}
B(z)=D(z) +\Pi(z)B
\end{displaymath}
with
\begin{equation}
\label{AG2B}
B=A+2G.
\end{equation}

Let $V(z)$ denote the minimum-phase spectral factor of
\begin{equation}
V(z)V(z^{-1})'=\Phi(z) := \Phi_+(z) + \Phi_+(z^{-1})' .
\end{equation}
From \cite[Chapter 6]{LPbook}, $V(z)$ has a realization of the form
\begin{equation}
V(z)=H(zI-F)^{-1}K + R,
\end{equation}
and following \eqref{lemma}, it can be expressed as
\begin{equation}
\label{ }
V(z)=A(z)^{-1}\Sigma(z)R,
\end{equation}
where
\begin{equation}
\label{Sigma(z)}
\Sigma(z)=D(z)+\Pi(z)\Sigma 
\end{equation}
with 
\begin{equation}
\label{Sigma}
\Sigma = A+KR^{-1}. 
\end{equation}

From stochastic realization theory \cite[Chapter 6]{LPbook} we obtain 
\begin{align}
  K  & =(G-FPH')(R')^{-1}  \label{K}\\
  RR'  &  = I-HPH' \label{R}
\end{align}
where $P$ is the unique minimum solution of the algebraic Riccati equation
\begin{equation}
\label{Riccati}
P=FPF' + (G-FPH')(I-HPH')^{-1}(G-FPH')'.
\end{equation}
From \eqref{F}, \eqref{Sigma}, \eqref{K} and \eqref{R}, \eqref{Riccati} can be written
\begin{equation}
\label{AREmod}
P=\Gamma (P-PH'HP)\Gamma' +GG' .
\end{equation}
where
\begin{equation}
\label{Gamma}
\Gamma=J-\Sigma H.
\end{equation}

The article \cite{CLtac} demonstrates that $G$ can be expressed as $u + U(\Sigma + \Gamma PH')$, with $u$ and $U$ being fully determined by the interpolation data \eqref{interpolation}. The  analytic interpolation problem involves determining the values of $(A,B)$ based on the given interpolation data \eqref{interpolation} and a specific matrix polynomial $\Sigma(z)$.   

In \cite{CLtac} we conclude that the conditions for the existence of solutions to this problem  only rely on the interpolation data. If the solution exists, it is also shown in \cite{CLtac}  that the {\em Covariance Extension Equation (CEE)}
\begin{subequations}\label{PgCCE}
	\begin{equation} \label{CEE}
		P=\Gamma (P-PH'HP)\Gamma' +G(P)G(P)' 
	\end{equation}
	 with
	\begin{equation}\label{g(P)} 
		G(P)= u +U(\Sigma + \Gamma PH') ,
	\end{equation}
\end{subequations}
 has a unique symmeric solution $P\geq 0$ with the property that $HPH'<1$. Additionally, for every $\Sigma$, there exists a unique solution to the analytic interpolation problem, which is expressed as follows:

\begin{subequations}\label{Psigma2ab}
	\begin{equation}\label{a}
		A=(I-U)(\Gamma PH'+\Sigma)-u
	\end{equation}
	\begin{equation}\label{b}
		B=(I+U)(\Gamma PH'+\Sigma)+u
	\end{equation}
	\begin{equation}\label{rho}
		R=(I-HPH')^{\frac{1}{2}} .
	\end{equation}
\end{subequations}
 The equation \eqref{PgCCE} can be solved using a homotopy continuation approach, as described in \cite{CLtac}. 
 Distinct selections of matrix polynomial $\Sigma(z)$ yield different feasible solutions.

\section{ Computational examples}\label{sec:applications}
\subsection{Example 1}

Consider a simple case where $x_0,y_0,x_1,y_1\in \mathbb{H}$ are given as follows:
\begin{equation}
    x_0=\frac{(s+12)(s-7)}{(s+1.5)(s+4.2)},\quad y_0=\frac{(s-2)(s-1)}{(s+2.5)(s+7.2)}
\end{equation}
\begin{equation}
    x_1=\frac{(s+3.6)(s-8)}{(s+4.7)(s+5.1)},\quad y_1=\frac{(s-1.3)(s+4)}{(s+3.3)(s+2.4)}
\end{equation}
To visualize unstable poles of $p_\lambda$ when $\lambda$ varies from 0 to 1, we apply the transformation:
\begin{equation}\label{trans}
    z=\frac{1+s}{1-s}
\end{equation}
which maps $\mathbb{C}_{-}$ to $\mathbb{D}$ and $\mathbb{C}_{+}$ to $\mathbb{D}^{C}$. Hence, a system is considered stable when all of its poles are located within the unit circle. Fig.~\ref{beforestabilization1} illustrates the positions of all poles of $p_\lambda$ as $\lambda$ ranges from 0 to 1 in increments of 0.1.

\begin{figure}[htp]
    \centering
    \includegraphics[width=1\linewidth]{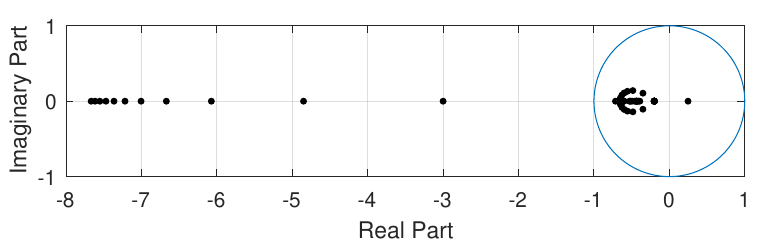}
    \caption{Poles of $p_\lambda$ before stabilization}
    \label{beforestabilization1}
\end{figure}

From Fig.~\ref{beforestabilization1}, some systems are unstable. To stabilize them, we observe that $x_0y_1-x_1y_0$ has zeros at $s_1=6.8652$ and $s_2=1$ in $\mathbb{C_+}$. To achieve stability, we need interpolation conditions:
\begin{equation}
    \left(\frac{\delta_{1}}{\delta_{0}}\right)(s_1)=\left(\frac{y_1}{y_0}\right)(s_1),\quad\left(\frac{\delta_{1}}{\delta_{0}}\right)(s_2)=\left(\frac{y_1}{y_0}\right)(s_2)
\end{equation}
By using the Möbius transformation \eqref{mob}, the problem then becomes to find a Carath\'eodory function $f(z)$ satisfying:
\begin{equation}
    f\left(\frac{1-s_1}{1+s_1}\right)=\sqrt{\left(\frac{y_1}{y_0}\right)(s_1)},\quad f\left(\frac{1-s_2}{1+s_2}\right)=\sqrt{\left(\frac{y_1}{y_0}\right)(s_2)}
\end{equation}
This is a special case of the analytic interpolation problem with $\ell=1,n_1=n_2=1$, called Nevanlinna-Pick interpolation problem. According to Proposition 5 in \cite{CLtac}, it has solutions. We choose $\Sigma(z)=z-0.5$. After calculation, we obtain:
\begin{equation}
\frac{\delta_1}{\delta_{0}}=\frac{93.342(s+0.1608)(s+0.1606)}{s^2+65.01s+1057}
\end{equation}

After stabilization, the poles of
\begin{equation}\label{tranfer}
	p_\lambda(s)(1+k(s)p_\lambda(s))^{-1}
\end{equation}
 where $\lambda$ changing from 0 to 1 at intervals of 0.1, are depicted in Fig.~\ref{afterstabilization1}. These poles are situated in $\mathbb{D}$, thereby confirming the system's stability.

\begin{figure}[htp]
    \centering
    \includegraphics[width=0.8\linewidth]{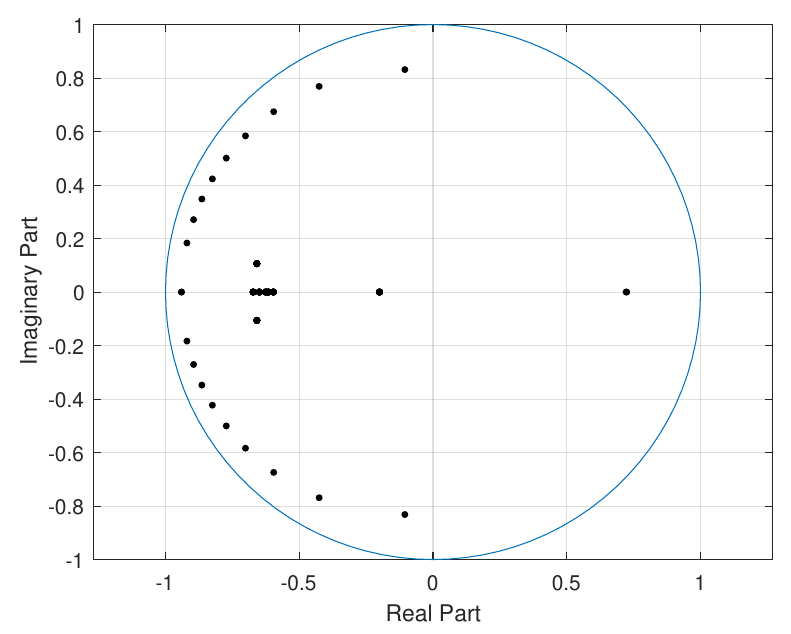}
    \caption{Poles of feedback systems after stabilization}
    \label{afterstabilization1}
\end{figure}

To demonstrate that different choices of $\Sigma(z)$ produce different feasible solutions, we vary the zero $z_0$ of $\Sigma(z)$ from 0 to 1. Fig.~3 shows the results for $z_0=0,0.2,0.4,0.6,0.8,0.99$ respectively, indicating that the solution changes with different $\Sigma(z)$.

\begin{figure}[htbp]
	\centering
	
		\begin{minipage}[t]{0.45\linewidth}
			\centering
			\includegraphics[width=1\linewidth]{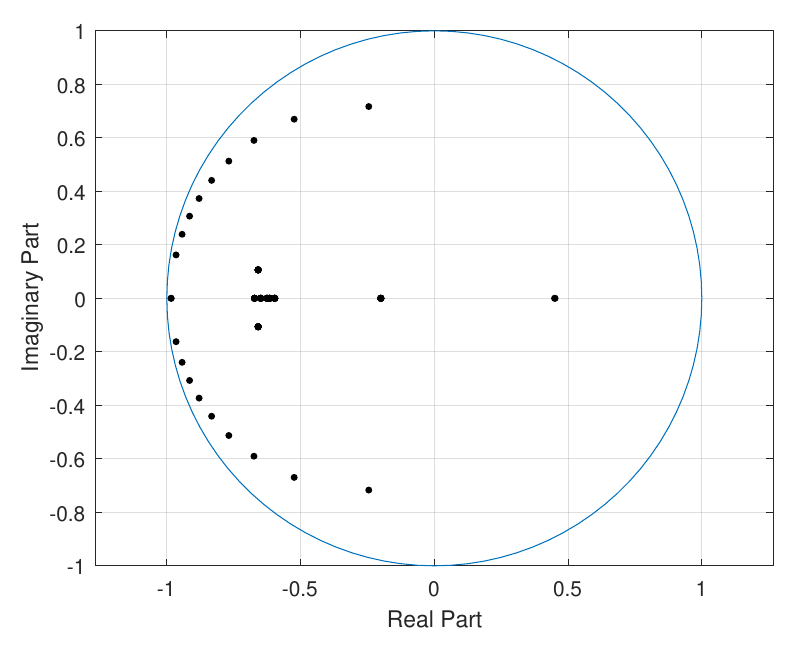}
		\end{minipage}%
		\begin{minipage}[t]{0.45\linewidth}
			\centering
			\includegraphics[width=1\linewidth]{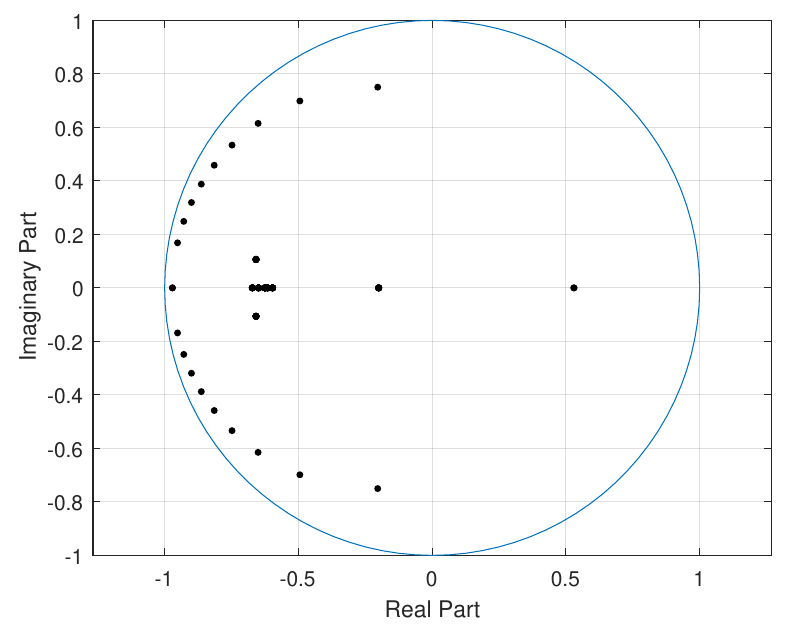}
		\end{minipage}%

		\begin{minipage}[t]{0.45\linewidth}
			\centering
			\includegraphics[width=1\linewidth]{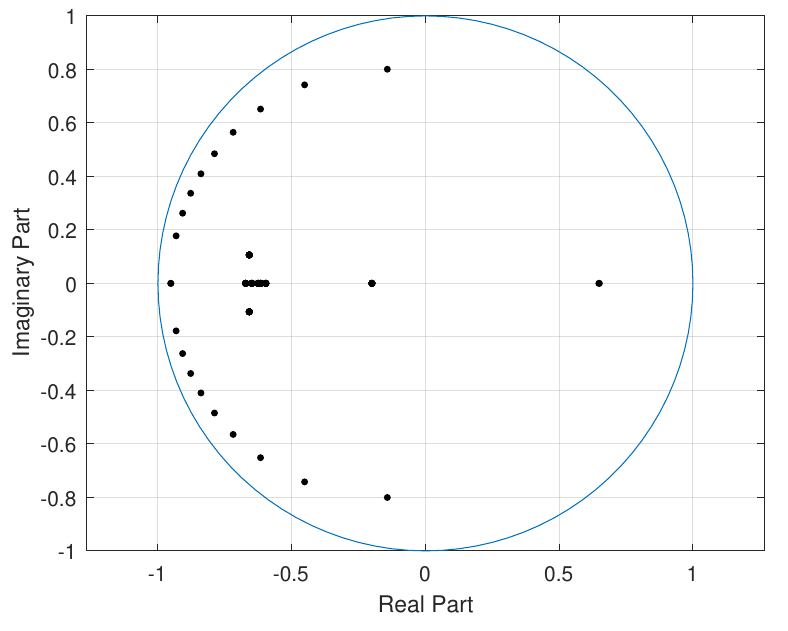}
		\end{minipage}
		\begin{minipage}[t]{0.45\linewidth}
			\centering
			\includegraphics[width=1\linewidth]{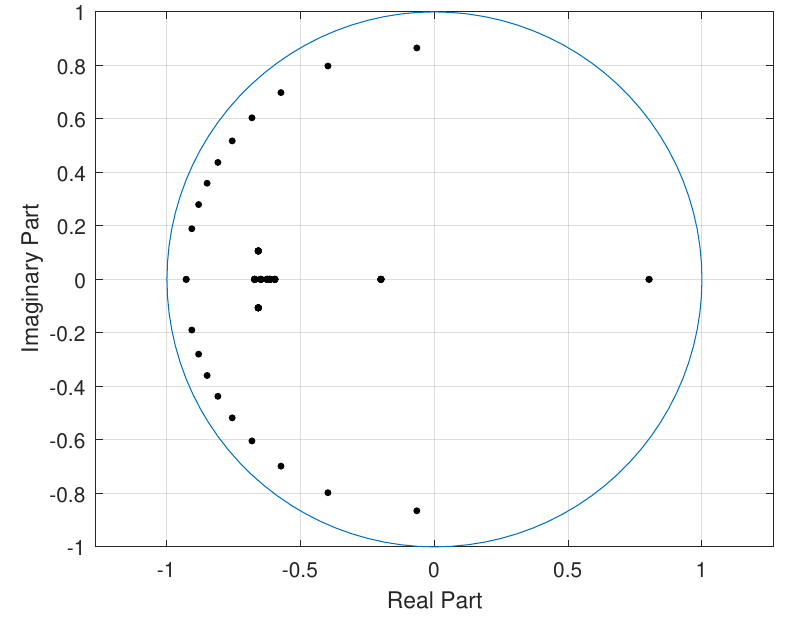}
		\end{minipage}

		\begin{minipage}[t]{0.45\linewidth}
			\centering
			\includegraphics[width=1\linewidth]{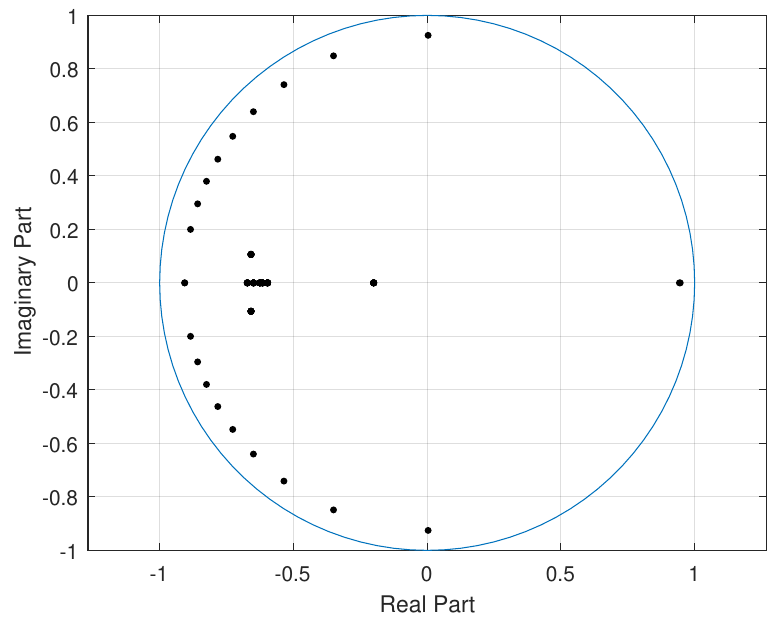}
		\end{minipage}
		\begin{minipage}[t]{0.45\linewidth}
			\centering
			\includegraphics[width=1\linewidth]{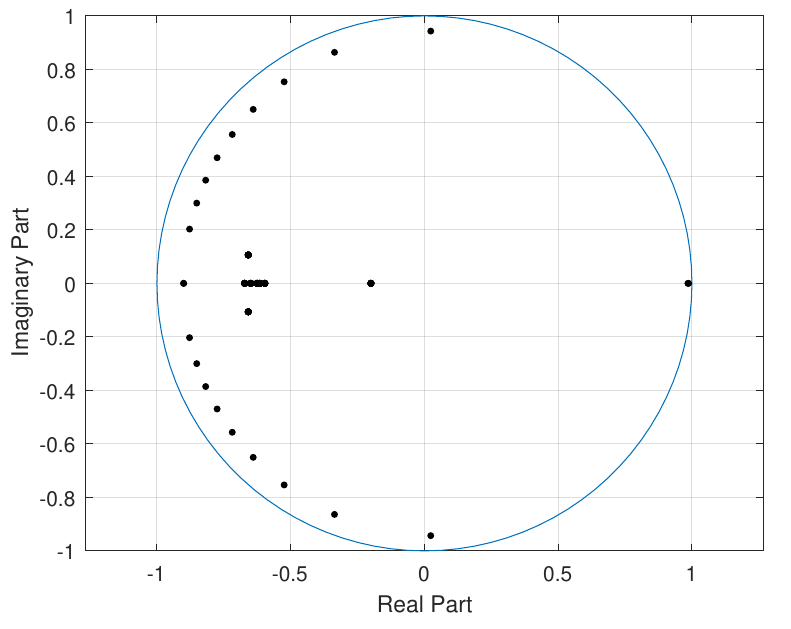}
		\end{minipage}
	\centering
	\label{sigma_zero1}
	\caption{Poles of the stabilized system with diferent $\Sigma(z)$}
\end{figure}

\subsection{Example 2}
Now, let's consider more complex systems that include derivative constraints:
\begin{equation}
    x_0=\frac{(s+0.7)(s+-0.1)}{(s+0.4)(s+0.9)},\quad y_0=\frac{(s-1)^2}{(s+0.5)(s+1.8)}
\end{equation}
\begin{equation}
    x_1=\frac{2(s+1.7)(s-0.3)}{(s+0.9)(s+1.4)},\quad y_1=\frac{(s-1)^2}{(s+1.2)(s+0.8)}
\end{equation}
It is evident that when $\lambda$ varies from 0 to 1, $p_\lambda$ has poles at 1, indicating instability in all systems.

By calculation, $x_0y_1-x_1y_0$ has zeros at $s_1=1$ and $s_2=357/604$  with multiplicities 2 and 1, respectively. This implies that we need to find $\delta_0$ and $\delta_1$ satisfying:
\begin{align}
    &\frac{\delta_{1}}{\delta_{0}}(s_1)=\frac{x_1}{x_0}(s_1),\quad\left(\frac{\delta_{1}}{\delta_{0}}\right)'(s_1)=\left(\frac{x_1}{x_0}\right)'(s_1)\\
    &\frac{\delta_{1}}{\delta_{0}}(s_2)=\frac{y_1}{y_0}(s_2)
\end{align}
Using the Möbius transformation \eqref{mob}, the problem is reduced to finding a Carath\'eodory function $f(z)$ satisfying:
\begin{align}
    &f(0)=\sqrt{\frac{x_1}{x_0}(1)},\quad f'(0)=-\frac{\left(\frac{x_1}{x_0}\right)'(1)}{f(0)}\\
    &f\left(\frac{1-s_2}{1+s_2}\right)=\sqrt{\frac{y_1}{y_0}(s_2)}
\end{align}
This is an analytic interpolation problem with derivative constraint. Here, we choose $\Sigma(z)=z(z-0.1)$. By calculation, we obtain:
\begin{equation}
    \frac{\delta_1}{\delta_0}=\frac{0.53942 (s+3.69)^2 (s+0.1353)^2}{(s^2 + 0.917s + 1.34)^2}
\end{equation}
Because there are three interpolation constraints, we end up with an $f(z)$ of degree 2 and a $\Delta_1/\Delta_0$ of degree 4. Fig.~\ref{example2} shows the poles of \eqref{tranfer} with $\lambda$ changing from 0 to 1 at intervals of 0.1. Since all poles are in $\mathbb{D}$, all systems are stable.

\begin{figure}[htp]
    \centering
    \includegraphics[width=0.8\linewidth]{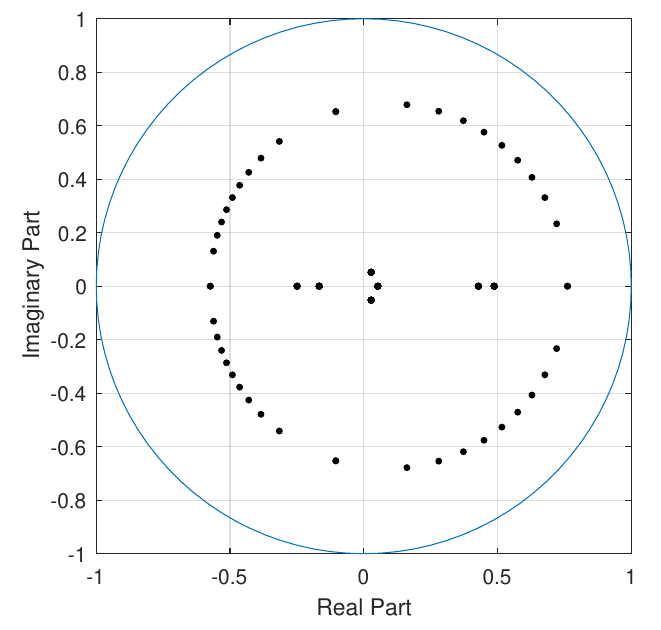}
    \caption{Poles of feedback systems after stabilization}
    \label{example2}
\end{figure}
\subsection{Example 3}
We now consider a simple MIMO simultaneous stabilization problem with
\begin{equation}
\begin{split}
     N_0&=\left[\begin{matrix}
	1.1&1.9\\
	2.9&1.1
\end{matrix}\right]\quad D_0=\left[\begin{matrix}
\frac{s-2.2}{s+5.8}&1\\
3&\frac{s-2.6}{s+9.7}
\end{matrix}\right]\\
N_1&=\left[\begin{matrix}
	1&2\\
	4&3
\end{matrix}\right]\quad D_1=\left[\begin{matrix}
	\frac{s-3.3}{s+2.1}&1\\
	6&\frac{s-7.8}{s+0.9}
\end{matrix}\right]
\end{split}   
\end{equation}

When $\lambda$ varies from 0 to 1, there are many unstable poles. To show the poles more clearly, we apply the transformation \eqref{trans}. Fig.~\ref{beforestabilization3} shows  all poles of $p_\lambda$  when $\lambda$ varies from 0 to 1 at
intervals of 0.1. 
\begin{figure}[htp]
	\centering
	\includegraphics[width=1\linewidth]{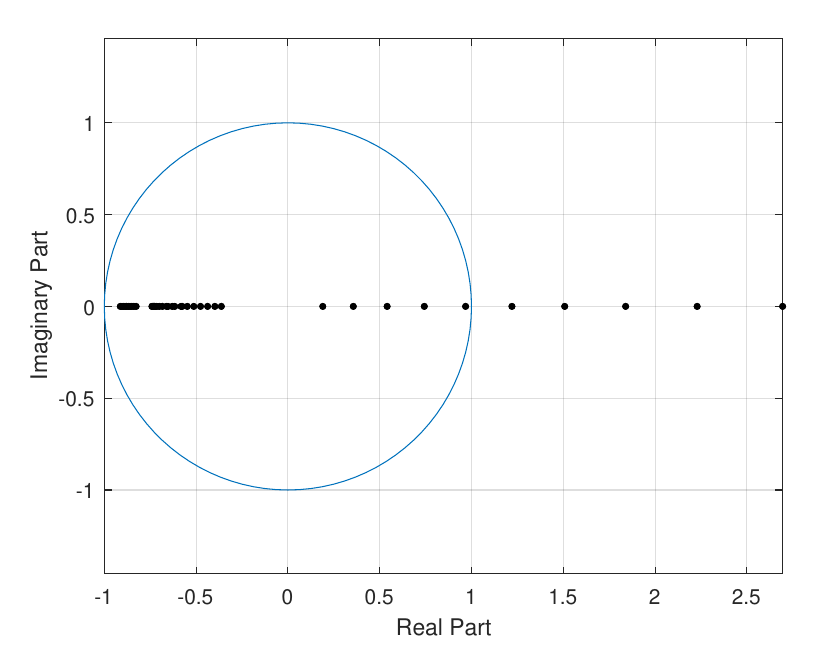}
	\caption{Poles of $P_\lambda$ before stabilization}
	\label{beforestabilization3}
\end{figure}

From Fig.~\ref{beforestabilization3}, we can see there are some systems that are not stable. Using the method in this paper, we first observe that $\det(M(s))$ has two zeros at $s_1=17.21$ and $s_2=0.9769$ in $\mathbb{C_+}$. To stabilize the systems, we  need the interpolation conditions \eqref{delta}, which yield
\begin{equation}
    (\Delta_0(s_1))^{-1}\Delta_1(s_1)=M_1,~(\Delta_0(s_2))^{-1}\Delta_1(s_2)=M_2,
\end{equation}
using the M{\"o}bius transformation \eqref{mob}, and since the Hermitian parts of $M_1^{1/2}, M_2^{1/2}$ are positive, the problem is then reduced to the analytic interpolation problem with $n_1=n_2=1$. The interpolation constraints are
\begin{equation}
\begin{split}
    (\Delta_0^{-1}\Delta_1)^{1/2}(\frac{1-s_1}{1+s_1})&=M_1^{1/2}\\
    (\Delta_0^{-1}\Delta_1)^{1/2}(\frac{1-s_2}{1+s_2})&=M_2^{1/2}
\end{split}
\end{equation}
By the theory in \cite{CLtac}, there exists solutions. Here we choose $\Sigma=[0.3~0;0~0.5]$ and get
\begin{equation}
    ((\Delta_0(s))^{-1}\Delta_1(s))^{1/2}=\begin{bmatrix}
        K_{11} & K_{12} \\
        K_{21} & K_{22} \\
    \end{bmatrix}
\end{equation}
$$
K_{11}=\frac{1.091 s^2 + 61.43 s + 857.6}{s^2 + 55.61 s + 745.3}
$$
$$
K_{12}=\frac{0.1289 s^2 + 25.88 s + 754.7}{s^2 + 55.61 s + 745.3}
$$
$$
K_{21}=\frac{-0.5455 s^2 -24.62 s -236.4}{s^2 + 55.61 s + 745.3}
$$
$$
K_{22}=\frac{0.4219 s^2 + 59.38 s + 1474}{s^2 + 55.61 s + 745.3}
$$

After stabilization, the poles of
\begin{equation}\label{transfer_m}
	P_\lambda(s)(I+K(s)P_\lambda(s))^{-1}
\end{equation}
where $\lambda$ varying from 0 to 1 at intervals of 0.1 are depicted in Fig.~\ref{afterstabilization3}. Given that all poles reside in $\mathbb{D}$, the compensator $K(s)$ can stabilize all these systems.

\begin{figure}[htp]
	\centering
	\includegraphics[width=1\linewidth]{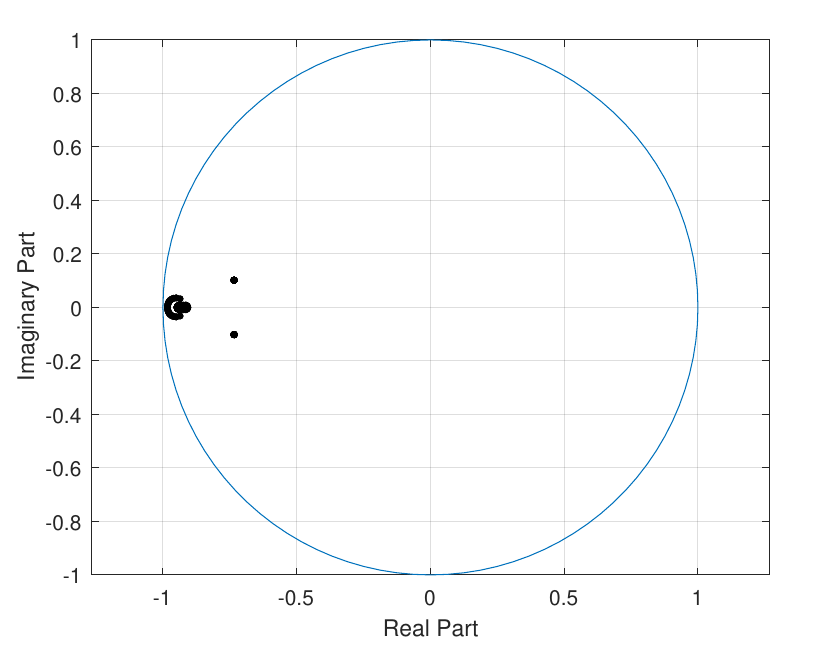}
	\caption{Poles of feedback systems after stabilization}
	\label{afterstabilization3}
\end{figure}

To demonstrate that different choices of $\Sigma$ produce different feasible solutions, we take $\Sigma$ to be
\begin{equation}
    \begin{split}
        a &= \begin{bmatrix} -0.1 & -0.9 \\ 0.4 & -0.6 \end{bmatrix}\qquad
b = \begin{bmatrix} 0.4 & 0.1 \\ 0.5 & 0.4 \end{bmatrix} \\
c &= \begin{bmatrix} 0.2 & 0.35 \\ 0.6 & 0.4 \end{bmatrix} ~~~\qquad
d = \begin{bmatrix} -0.8 & 0.1 \\ 0.6 & -0.2 \end{bmatrix} \\
e &= \begin{bmatrix} -0.65 & 0.22 \\ 0.8 & -0.2 \end{bmatrix} \qquad
f = \begin{bmatrix} 0.8 & -0.33 \\ 0.9 & 0.1 \end{bmatrix}.
    \end{split}
\end{equation}
respectively. Fig.~7 shows the corresponding results, indicating that the solution changes with different $\Sigma$.

\begin{figure}[htbp]
	\centering
		\begin{minipage}[t]{0.45\linewidth}
			\centering
			\includegraphics[width=1\linewidth]{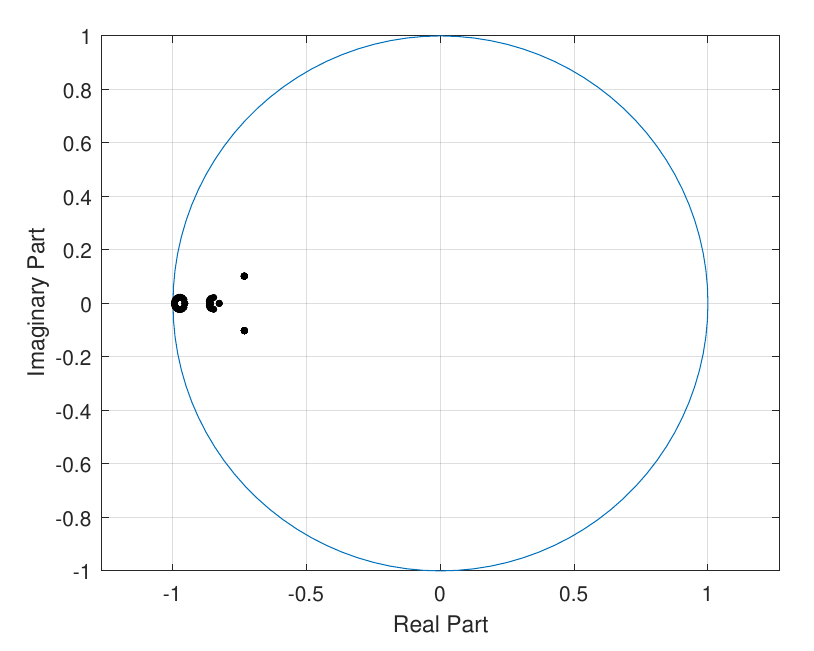}
		\end{minipage}
		\begin{minipage}[t]{0.45\linewidth}
			\centering
			\includegraphics[width=1\linewidth]{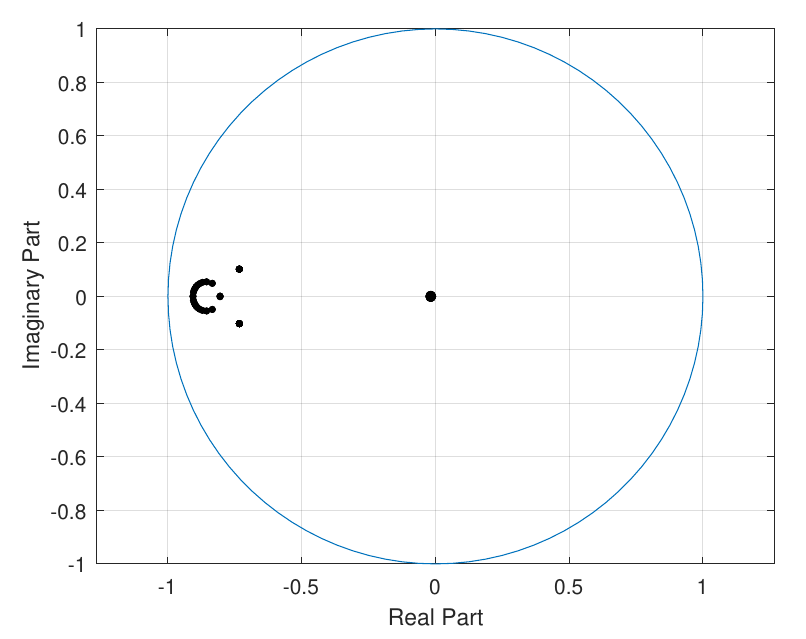}
		\end{minipage}
	
		\begin{minipage}[t]{0.45\linewidth}
			\centering
			\includegraphics[width=1\linewidth]{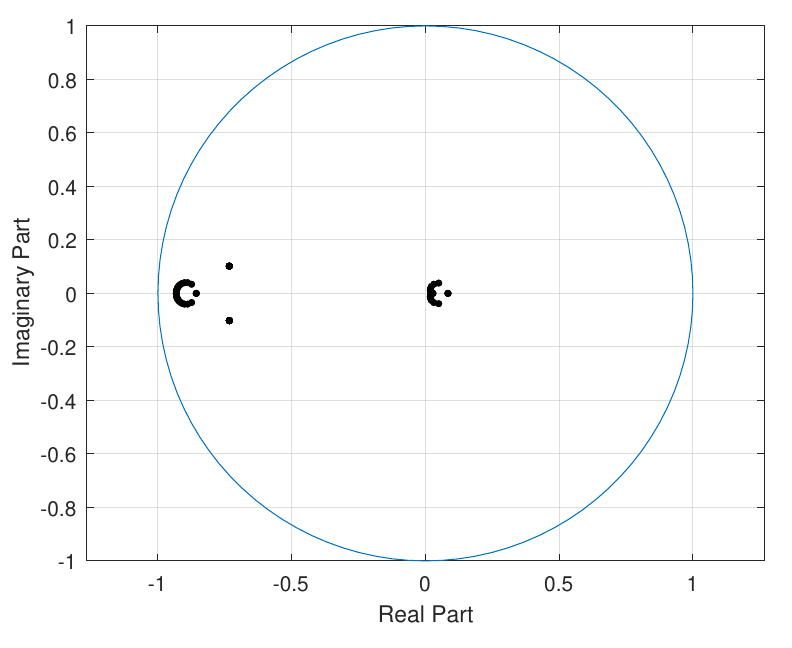}
		\end{minipage}
		\begin{minipage}[t]{0.45\linewidth}
			\centering
			\includegraphics[width=1\linewidth]{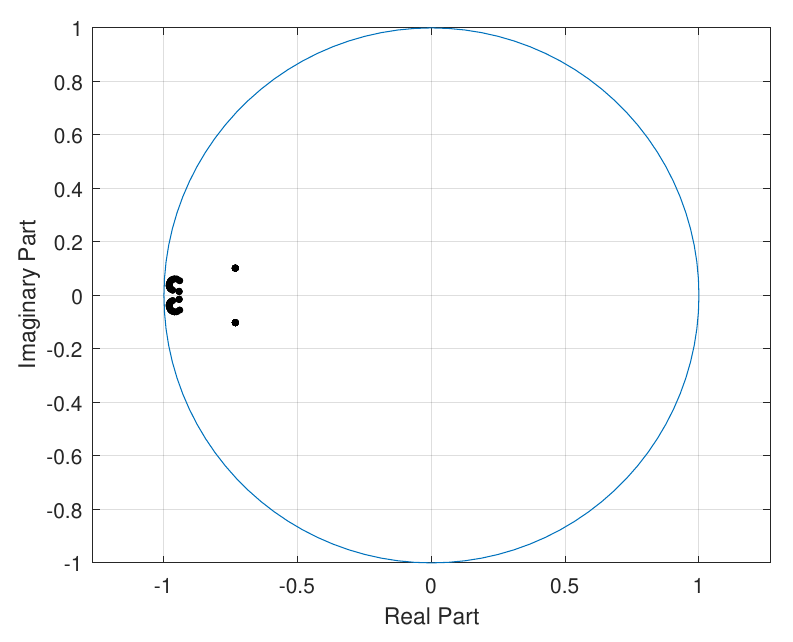}
		\end{minipage}
	\begin{minipage}[t]{0.45\linewidth}
		\centering
		\includegraphics[width=1\linewidth]{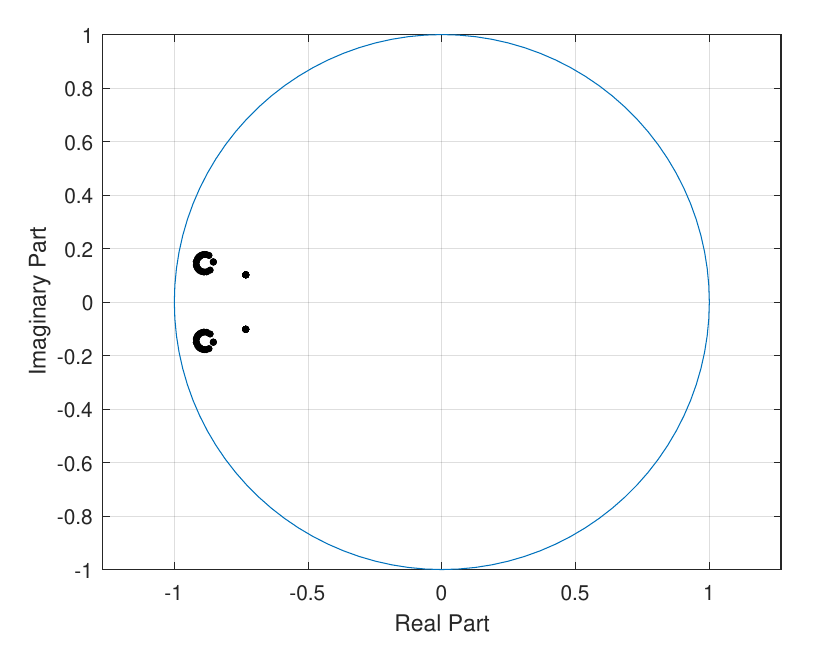}
	\end{minipage}
	\begin{minipage}[t]{0.45\linewidth}
		\centering
		\includegraphics[width=1\linewidth]{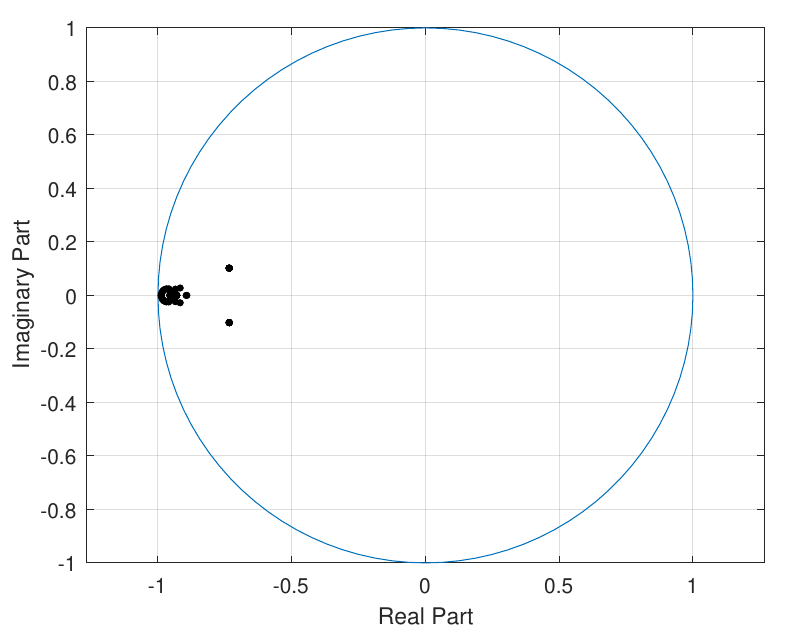}
	\end{minipage}
	\centering
	\label{sigma_zero}
	\caption{Poles of the stabilized system with diferent $\Sigma$}
\end{figure}
\subsection{Example 4}
Next we consider a more complex system which includes complex unstable zeros. Suppose
\begin{equation*}
   N_0=\frac{1}{z^2+2z+10}\left[\begin{matrix}
	3(z^2-2z+1)&5(z^2+2z+10)\\
	4(z^2+2z+10)&2(z-4)(z-2)
\end{matrix}\right] 
\end{equation*}
\begin{equation*}
    D_0=\frac{1}{z^2+2z+10}\left[\begin{matrix}
2(z-2)(z-3)&-(z^2+2z+10)\\
3(z^2+2z+10)&(z-2)^2
\end{matrix}\right]
\end{equation*}

and
$$
N_1=\frac{1}{z^2+2z+15}\left[\begin{matrix}
	(z-1)(z+1)&6(z^2+2z+15)\\
	7(z^2+2z+15)&(z-8)(z-1)
\end{matrix}\right]
$$
$$
D_1=\frac{1}{z^2+2z+15}\left[\begin{matrix}
	(z-6)(z+2)&2(z^2+2z+15)\\
	-(z^2+2z+15)&3(z-5)(z-9)
\end{matrix}\right]
$$

There are many unstable poles when $\lambda$ varies on the interval $[0,1]$. In Fig.~\ref{beforestabilization4} we show the poles  of $p_\lambda$  as $\lambda$ varies from 0 to 1 at intervals of 0.1. Unlike the previous example, there are many complex poles in this example.

\begin{figure}[htp]
	\centering
	\includegraphics[width=1\linewidth]{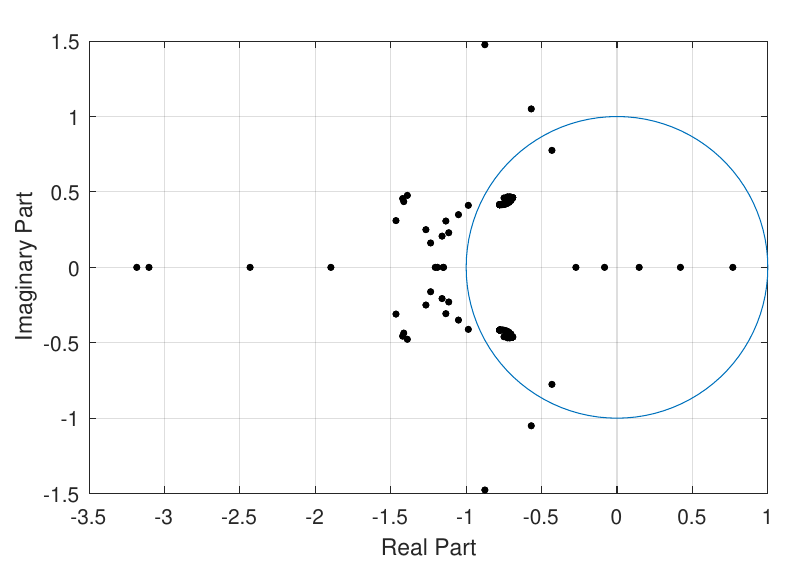}
	\caption{Poles of $P_\lambda$ before stabilization}
	\label{beforestabilization4}
\end{figure}

Using the method in this paper, we first observe that $\det(M(s))$ has three zeros at $s_1=0.2322$, $s_2=0.9862-3.5291i$, $s_3=0.9862+3.5291i$  in $\mathbb{C_+}$. To make the systems stable, we therefore need the interpolation conditions \eqref{delta}, which yield
\begin{equation}
    \begin{split}
        (\Delta_0(s_1))^{-1}&\Delta_1(s_1)=M_1,\\
        (\Delta_0(s_2))^{-1}\Delta_1(s_2)=M_2,&
(\Delta_0(s_3))^{-1}\Delta_1(s_3)=M_3
    \end{split}
\end{equation}
using the M{\"o}bius transformation \eqref{mob} and  since the  Hermitian  parts of $M_1^{1/2},M_2^{1/2},M_3^{1/2}$ are positive,
the problem is then reduced to  the analytic interpolation problem with the  interpolation conditions 
\begin{equation}
\begin{split}
    (\Delta_0^{-1}\Delta_1)^{1/2}(\frac{1-s_1}{1+s_1})&=(M_1)^{1/2}\\
    (\Delta_0^{-1}\Delta_1)^{1/2}(\frac{1-s_2}{1+s_2})&=(M_2)^{1/2}\\(\Delta_0^{-1}\Delta_1)^{1/2}(\frac{1-s_3}{1+s_3})&=(M_3)^{1/2}
\end{split}
\end{equation}

Here we choose 
\begin{equation}
\Sigma = \begin{bmatrix} 0.4 & 0.2 \\ 0.3 & -0.5 \\ 0.8 & -0.1 \\ 0.6 & -0.2 \end{bmatrix}.
\end{equation}

Then we get the result
\begin{equation}
    (\Delta_0(s)^{-1}\Delta_1(s))^{1/2}=\begin{bmatrix}
    K_{11}&K_{12}\\
    K_{21}&K_{22}
\end{bmatrix}
\end{equation}
$$K_{11}=\frac{s^4 + 0.923 s^3 + 0.6752 s^2 + 0.1567 s + 0.02193}{s^4 + 0.3846 s^3 + 0.6211 s^2 + 0.1852 s + 0.0186}$$
$$K_{12}=\frac{-0.325 s^4 + 1.821 s^3 + 0.3929 s^2 - 0.0064 s - 0.0111}{s^4 + 0.3846 s^3 + 0.6211 s^2 + 0.1852 s + 0.0186}$$
$$K_{21}=\frac{-0.5581 s^4 + 0.9026 s^3 - 0.2138 s^2 - 0.01 s - 0.03345}{s^4 + 0.3846 s^3 + 0.6211 s^2 + 0.1852 s + 0.0186}$$
$$K_{22}=\frac{0.2714 s^4 + 3.214 s^3 + 1.429 s^2 + 0.1429 s + 0.017}{s^4 + 0.3846 s^3 + 0.6211 s^2 + 0.1852 s + 0.0186}$$

The poles of \eqref{transfer_m} where $\lambda$ varying from 0 to 1 at intervals of 0.1 are shown in Fig.~\ref{afterstabilization4}. Because all poles are in $\mathbb{D}$,  all feedback systems are stable.

\begin{figure}[htp]
	\centering
	\includegraphics[width=1\linewidth]{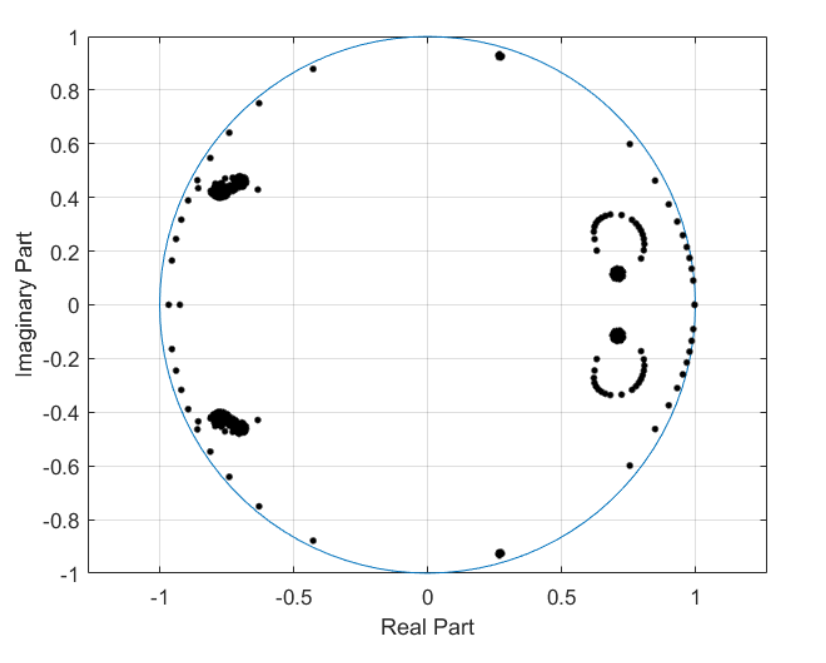}
	\caption{Poles of feedback systems after stabilization}
	\label{afterstabilization4}
\end{figure}

\section{conclusion}\label{sec:conclusion}

In this paper, we have explored the domain of simultaneous stabilization in control systems engineering, addressing both scalar and multivariable scenarios with a unified approach. Our investigation has led to the formulation of necessary and sufficient conditions for achieving stability across multiple plants using a single controller. By adapting Ghosh's interpolation framework to incorporate derivative constraints and employing a novel Riccati-type matrix equation approach, we have not only expanded the theoretical understanding of simultaneous stabilization but also provided a practical method for parameterizing solutions. Through a series of detailed numerical examples, we have demonstrated the efficacy and versatility of our method. In future work we shall allow derivative constraints for the multivariable scenario and consider a set of plants with different formats rather than linear combination of two distinct plants.

\bibliographystyle{IEEEtran}

\end{document}